\documentclass[11pt,a4paper,reqno]{amsart}
\usepackage[english]{babel}
\usepackage[applemac]{inputenc}
\usepackage{bbm}
\usepackage{cite}
\usepackage{float, verbatim}
\usepackage{amsmath}
\usepackage{amssymb}
\usepackage{amsthm}
\usepackage{amsfonts}
\usepackage{graphicx}
\usepackage{mathtools}
\usepackage{enumitem}
\usepackage[colorlinks = true, citecolor = black]{hyperref}
\usepackage{color}
\usepackage{tikz}
\usepackage{xcolor}
\usepackage{pgfplots}
\usepackage{caption}
\usepackage{subcaption}
\usepackage{enumerate}
\usepackage{autobreak}
\usepackage{url}

\newcommand{\ubox}{\overline{\dim_{\mathrm{B}}}}
\newcommand{\lbox}{\underline{\dim_{\mathrm{B}}}}
\newcommand{\nbox}{\dim_{\mathrm{B}}}

\newcommand{\Haus}{\dim_{\mathrm{H}}}

\newtheorem*{thm*}{Theorem}
\newtheorem*{conj*}{Conjecture}

\newtheorem{thm}{Theorem}[section]

\newtheorem{lma}[thm]{Lemma}
\newtheorem{cor}[thm]{Corollary}
\newtheorem{defn}[thm]{Definition}

\newtheorem{conj}[thm]{Conjecture}

\begin{document}
	
	\title{Multi-rotations on the unit circle}
	
	\author{Han Yu}
	\address{Han Yu\\
		School of Mathematics \& Statistics\\University of St Andrews\\ St Andrews\\ KY16 9SS\\ UK \\ }
	\curraddr{}
	\email{hy25@st-andrews.ac.uk}
	\thanks{}
	
	
	
	\date{}
	
	\dedicatory{}
	
	\begin{abstract}
		In this paper, we study multi-rotation orbits on the unit circle. We obtain a natural generalization of a classical result which says that orbits of irrational rotations on the unit circle are dense. It is possible to show that this result holds true if instead of iterating a single irrational rotation, one takes a multi-rotation orbit along a finitely recurrent sequence over finitely many different irrational rotations. We also discuss some connections between the box dimensions of multi-rotation orbits and Diophantine approximations. In particular, we improve a result by Feng and Xiong in the case when the rotation parameters are algebraic numbers.
	\end{abstract}
	
	\maketitle
	\allowdisplaybreaks
	\section{introduction}
	\subsection{Multi-rotation orbits: definitions and known results}
	Let $\alpha$ be an irrational number. We define its \emph{reduced fractional part} $\{\alpha\}\in (-0.5,0.5]$ to be the difference between $\alpha$ and its closest integer. Notice that when $\alpha$ is irrational, its closest integer is well defined. A classical result in number theory says that the sequence of fractional parts $\{n\alpha\}_{n\geq 1}$ is dense in $[-0.5,0.5].$ Naturally, one may consider rotations with more than one angles. More precisely, let $k\geq 1$ be an integer and let $\alpha_1,\dots,\alpha_k$ be $k$ irrational numbers such that $1,\alpha_1,\dots,\alpha_k$ are linearly independent over the field $\mathbb{Q}.$ We call numbers satisfying the above condition to be \emph{rationally independent}. This condition rules out some non-interesting cases for example when $k=2, \alpha_1=\alpha_2.$   Let $\omega\in \Lambda^\mathbb{N}$ be a sequence over digits $\Lambda=\{\alpha_1,\dots,\alpha_k\}.$ For each integer $i\geq 1$, let $\omega(i)$ be the $i$-th digit of $\omega$. We construct the following set
	\[
	E_\omega=\overline{\{x_n(\omega)\}_{n\in\mathbb{N}}},
	\]
	where we define the sequence $\{x_n(\omega)\}_{n\in\mathbb{N}}$ inductively as follows,
	\[
	x_0(\omega)=0, \forall i\geq 1, x_{i}(\omega)=\{x_{i-1}(\omega)+\omega(i)\}.
	\]
	We call $E_\omega$ the \emph{multi-rotation orbit with rotation parameters $\alpha_1,\dots,\alpha_k$ along $\omega.$ } In the case when $\omega\in\Lambda^k$ for an integer $k\geq 1,$ we can still define $x_i(\omega)$ for $i\in\{0,\dots,k\}$ in the above manner. In our notation for multi-rotation orbits and their closures we omit the dependence on $\alpha_1,\dots,\alpha_k.$ A more precise way would be writing $E_\omega(\alpha_1,\dots,\alpha_k)$ and $x_i(\omega,\alpha_1,\dots,\alpha_k)$ instead. However these notations are in some sense superfluous since the dependence on $\alpha_1,\dots,\alpha_k$ is always assumed.	When $k=2,$ it is conventional to call $E_\omega$ an $\alpha\beta$-set but we will not use this terminology in this paper.
	
	The study of $E_\omega$ dates to \cite{E61}. Unlike the single rotation case in which $E_\omega$ must be the interval $[-0.5,0.5]$, when $k=2$, one can find $\alpha_1,\alpha_2,\omega$ such that $E_\omega$ has arbitrarily small Hausdorff dimension. In particular, $E_\omega$ is nowhere dense, see \cite{K79}. Recently, in \cite{FX18}, it was proved that for $k=2$, any pairs of irrational numbers $\alpha_1,\alpha_2,$ and any sequence $\omega$, the difference set $E_\omega-E_\omega$ must contain intervals. In particular, this implies that $\lbox E_\omega\geq 0.5.$ See Section \ref{PRE} for definitions of dimensions. This result allows one to study affine embeddings between self-similar sets with more than one contraction ratios, see \cite{FX18} for more details. In this direction, one can ask whether or not $\lbox E_\omega=1.$
	\begin{conj}\label{feng}
		Let $k\geq 1$ be an integer and let $\Lambda=\{\alpha_1,\dots,\alpha_k\}$ be a set of $k$ rationally independent numbers.
		Let $\omega\in\Lambda^\mathbb{N}$ be an arbitrary sequence over $\Lambda$. Then we have $\lbox E_\omega=1$ where $E_\omega$ is the multi-rotation orbit with parameters in $\Lambda$ along $\omega$.
	\end{conj}
	
	 So far, the best result in this direction is \cite[Theorem 1.5]{FX18}.
	\begin{thm*}[\cite{FX18},Theorem 1.5 ]
		Let $k\geq 2$ be an integer and let $\Lambda=\{\alpha_1,\dots,\alpha_k\}$ be a set of $k$ rationally independent numbers.
			Let $\omega\in\Lambda^\mathbb{N}$ be an arbitrary sequence over $\Lambda$. If $k\geq 3$, then we have $\lbox E_\omega\geq 1/(k+1)$ and if $k=2$ then we have $\lbox E_\omega\geq 1/2.$
	\end{thm*}	
	\subsection{Multi-rotation orbits and recurrent sequences}
	Our first observation is that instead of an arbitrary sequence, it is enough to study sequences with additional structures, for example, being recurrent. In what follows, we use the term \emph{word} to mean a finite length sequence over a digit set. Given a digit set $\Lambda$ and a sequence $\omega\in\Lambda^{\mathbb{N}}$, a word $a$ in $\omega$ is an element of $\Lambda^{k}$ for an integer $k\geq 1$ and there is an number $i\geq 1$ such that for each $j\in\{1,\dots,k\}$, the $j$-th digit of $a$ is the same as the $(i+j)$-th digit of $\omega.$ The following definition of recurrent sequence is not commonly used in literatures, one can find more details in \cite[Chapter 1 Section 3]{F81}.  
	
	\begin{defn}\label{ReCur}
		Let $\Lambda$ be a finite set of digits. A sequence over $\Lambda$ is called recurrent if there exist a sequence $\{a_i\}_{i\in\mathbb{N}}$ of finite length words (we call them the building words for $\omega$) over digits $\Lambda$ such that $\omega$ can be obtained as follows. First, we take $\omega_0=a_0$. For each $i\geq 1$ we take $\omega_i=\omega_{i-1}a_i\omega_{i-1}.$ We write $\omega$ to be the limit $\lim_{i\to\infty} \omega_i.$ The first several digits of $\omega$ can be listed out as  follows,
		\[
		((a_0 a_1 a_0)a_2(a_0 a_1 a_0))a_3((a_0 a_1 a_0)a_2(a_0 a_1 a_0))a_4\dots.
		\] 
	\end{defn}
	
	One of the most famous examples of recurrent sequences is the Thue-Morse sequence: $$abbabaabbaababbabaababbaabbabaab\dots.$$
	More precisely, one can obtain the above sequence by starting with a digit $a$ and apply the substitution rule $a\to ab$ and $b\to ba$ inductively.
	
	\begin{thm}\label{Th1}
		Under the hypothesis of Conjecture \ref{feng} the followings are equivalent:
		\begin{itemize}
			\item{1}: $\lbox E_\omega=1$ for each sequence $\omega\in\Lambda^\mathbb{N}.$
			\item{2}: $\lbox E_\omega=1$ for each recurrent sequence $\omega\in\Lambda^\mathbb{N}.$
		\end{itemize}
	\end{thm} 
	
	The most important ingredient for proving the above result is Birkhoff's recurrent theorem for topological dynamical systems, see \cite[Chapter 2]{PY} and \cite[Chapter 1  Section 3,4,5]{F81}. Morally, this result allows us to concentrate on multi-rotation orbits along recurrent sequences. This seems to be not a big progress. One example that illustrates the strength of recurrence is the following theorem which can be viewed as a generalization of the classical irrational rotation result mentioned at the beginning of this section.

	\begin{defn}
		We say that $\omega$ is finitely recurrent if it is recurrent in the sense of Definition \ref{ReCur} and furthermore the building words $a_i,i\geq 0$ have a uniform bound of length. That is to say, there is an integer $M$ such that the length of each $a_i$ is at most $M.$ An equivalent statement would be that $\{a_i\}_{i\geq 0}$ is a finite set.
	\end{defn}
	
	\begin{thm}[Generalization of irrational rotation]\label{Th2}
		Under the hypothesis of Conjecture \ref{feng}, let $\omega\in\Lambda^\mathbb{N}$ be a finitely recurrent sequence. Then $\{x_i(\omega)\}_{i\geq 1}$ is dense in $[-0.5,0.5].$
	\end{thm}
	
	\subsection{Multi-rotation orbits and Diophantine approximation}
	Our next result improves \cite[Theorem 1.5]{FX18} when the rotation parameters $\alpha_1,\dots,\alpha_k$ are rationally independent algebraic numbers. 
	
	\begin{thm}\label{Th3}
		Under the hypothesis of Conjecture \ref{feng}, suppose further that $\alpha_1,\dots, \alpha_k$ are algebraic numbers, then we have for all $\omega\in \Lambda^{\mathbb{N}}$ that
		\[
		\lbox E_\omega\geq \frac{1}{k}.
		\]
	\end{thm} 
	It turns out that the study of box dimensions of multi-rotation orbits is closely related to Diophantine approximations of linear forms. We will discuss more about this connection in Section \ref{FUR}. It is natural to ask whether one can get rid of the algebraic condition in the above theorem. We will see (by Theorem \ref{LebAlmost}) that for Lebesgue almost all $k$-tuples $\alpha_1,\dots,\alpha_k$ and all $\omega\in\{\alpha_1,\dots,\alpha_k\}^{\mathbb{N}}$ we have
\[
\ubox E_\omega\geq \frac{1}{k}.
\]
	This still leaves the general question open. Note that the following conjecture is much weaker than Conjecture \ref{feng}. We pose it here because it looks very natural and should be much easier to prove than Conjecture \ref{feng}. 
	
		\begin{conj}
		Let $k\geq 2$ be an integer and let $\lambda=\{\alpha_1,\dots,\alpha_k\}$ be a set of $k$ rationally independent numbers.
		Let $\omega\in\Lambda^\mathbb{N}$ be an arbitrary sequence over $\Lambda$. Then we have $\lbox E_\omega\geq 1/k$.
	\end{conj}	
	\section{Preliminaries}\label{PRE}
	\subsection{Dimensions}
	We list here some basic definitions of dimensions mentioned in the introduction.  For more details, see \cite[Chapters 2,3]{Fa} and \cite[Chapters 4,5]{Ma1}. We shall use $N(F,r)$ for the minimal covering number of a set $F$ in $\mathbb{R}^n$ with balls of side length $r>0$. 
	
	\subsubsection{Hausdorff dimension}
	
	Let $g: [0,1)\to [0,\infty)$ be a continuous function such that $g(0)=0$. Then for all $\delta>0$ we define the following quantity
	\[
	\mathcal{H}^g_\delta(F)=\inf\left\{\sum_{i=1}^{\infty}g(\mathrm{diam} (U_i)): \bigcup_i U_i\supset F, \mathrm{diam}(U_i)<\delta\right\}.
	\]
	The $g$-Hausdorff measure of $F$ is
	\[
	\mathcal{H}^g(F)=\lim_{\delta\to 0} \mathcal{H}^g_{\delta}(F).
	\]
	When $g(x)=x^s$ then $\mathcal{H}^g=\mathcal{H}^s$ is the $s$-Hausdorff measure and Hausdorff dimension of $F$ is
	\[
	\Haus F=\inf\{s\geq 0:\mathcal{H}^s(F)=0\}=\sup\{s\geq 0: \mathcal{H}^s(F)=\infty          \}.
	\]
	\subsubsection{Box dimensions}
	The upper box dimension of a bounded set $F$ is
	\[
	\overline{\nbox} F=\limsup_{r\to 0} \left(-\frac{\log N(F,r)}{\log r}\right).
	\]
	Similarly the lower box dimension of $F$ is
	\[
	\lbox F=\liminf_{r\to 0} \left(-\frac{\log N(F,r)}{\log r}\right).
	\]
	If the limsup and liminf are equal, we call this value the box dimension of $F$ and we denote it as $\nbox F.$
	\subsection{The fractional part symbol: an apology}
	We have mentioned at the beginning of this paper that for an irrational number $\alpha$, we use $\{\alpha\}$ to denote its reduced fractional part. Formally it is defined as follows,
	\[
	\{\alpha\}=\alpha-N_\alpha,
	\]
	where $N_\alpha$ is such that
	\[
	|N_\alpha-\alpha|=\min\{N\in\mathbb{N}: |N-\alpha|  \}.
	\]
	It is unfortunate that we use $\{.\}$ to denote a set as well. We hope that no confusion would arise with this abuse of notation. Following a tradition in Diophantine approximation, the absolute value $|\{\alpha\}|$ will be written as
	$
	\|\alpha\|
	$ 
	in this paper.
	\subsection{A Diophantine approximation result by Schmidt}
	It was shown in \cite[Theorem 2]{S80} the following theorem.
	
	\begin{thm}\label{SchmidtThm}
		Let $k\geq 2$ be an integer and $\alpha_1,\dots,\alpha_k$ be $k$ rationally independent algebraic numbers. Then for each $\delta>0$ there are only finitely many $k$-tuples of integers $q_1,\dots,q_k$ such that
		\[
		\|q_1\alpha_1+\dots+q_k\alpha_k\|(q_1q_2\dots q_n)^{1+\delta}\leq 1.
		\] 
	\end{thm}

	\section{Proofs of Theorem \ref{Th1} and \ref{Th2}}
	
	For any finite set $\Lambda$ of digits,  we can consider $\Lambda^{\mathbb{N}}$ as a metric space by defining the distance between $\omega',\omega\in\Lambda^\mathbb{N}$ to be $2^{-n(\omega,\omega')}$ where
	\[
	n(\omega,\omega')=\min\{k\in\mathbb{N}: \omega(k)\neq \omega'(k)  \}.
	\]
	In other words, two sequences are close to each other if they share a lot of initial digits in common. Let $S:\Lambda^{\mathbb{N}}\to\Lambda^{\mathbb{N}}$ be the left shift, namely, for each $\omega=\omega(1)\omega(2)\dots$ we have
	\[
	S(\omega)=\omega(2)\omega(3)\dots.
	\]
	The following result shows that to consider $E_\omega$ it is enough to consider any limit point of the shifted orbit $\{S^{i}(\omega)\}_{i\geq 0}.$ 
	\begin{lma}\label{LimitShift}
		Let $\omega\in\Lambda^\mathbb{N}$ be any sequence and let $\omega^*$ be a limit point of $\{S^i(\omega)\}_{i\geq 1}$. Then there exists a number $c(\omega^*)\in [-0.5,0.5]$ such that $c(\omega^*)+E_{\omega^*}\mod 1\subset E_\omega.$
	\end{lma}
	\begin{proof}
		For each integer $n\geq 1$, we write $\omega_1^n$ for the block of the first $n$ digits of $\omega.$ Since $\omega^*$ is a limit point of $\{S^i(\omega)\}_{i\geq 1}$, we can find sequences $i_k,j_k$ with $j_k\to\infty$ such that \[\omega_{i_k+1}^{i_k+j_k}={\omega^*}_{1}^{j_k}.\]
		The infinite sequence $\{x_{i_k+1}(\omega)\}_{k\in\mathbb{N}}$ has at least one limit point $c=c(\omega)\in [-0.5,0.5].$ By taking a subsequence if necessary we assume that \[\lim_{k\to\infty} x_{i_k+1}=c.\] We see that for each $k\in\mathbb{N}$, the points \[c_k, c_k+\omega(i_k+1), c_k+\omega(i_k+2), \dots, c_k+\omega(i_k+j_k)\] are contained in $E_{\omega}.$ Once we replace $\omega_{i_k+1},\dots,\omega_{i_k+j_k}$ with ${\omega^*}_{1},\dots,{\omega^*}_{j_k}$ and take the limit with $k\to\infty$ we see that $c+E_{\omega^*}\subset E_{\omega}$ as required.
	\end{proof}
	
	We can now finish the proof of Theorem \ref{Th1}.
	\begin{proof}[Proof of Theorem \ref{Th1}]
		By Birkhoff's recurrent theorem and the above lemma we see that $E_\omega$ contains a translated copy of $E_{\omega'}$ for a recurrent sequence $\omega'$. Thus this proof of Theorem \ref{Th1} is completed.
	\end{proof}

	\begin{thm}[$0$ is a returning position]\label{Zero}
		Let $\omega\in\Lambda^\mathbb{N}$ be a recurrent sequence then $0$ is a non-isolated limit point of $\{x_i(\omega)\}_{i\geq 1}.$ 
	\end{thm}
	\begin{proof}
		Let $\omega=\lim_{i\to\infty} \omega_i$ as discussed above in Definition \ref{ReCur} with finite sequences $a_0,a_1\dots.$ Consider the sequences $v_i=\omega_{i-1}a_i$. Observe that $v_i$ is a prefix of the sequence $\omega$ for each $i\geq 1.$ Denote the length of $v_i$ as $L_i$. Since $\omega_i\to\omega$, we see that for each limit point $c$ of the sequence of numbers $\{x_{L_i}(\omega)\}_{i\geq 1}$, by Lemma \ref{LimitShift}, $c+E_{\omega}\mod 1\subset E_{\omega}.$ Thus if $c$ is an irrational number then $E_{\omega}$ contains an entire orbit of an irrational rotation and therefore $E_\omega$ must be the whole unit interval. Otherwise suppose that $c=p/q$ is a rational number. Without loss of generality we shall assume that $(p,q)=1.$ Then we see that
		\[
		E_\omega\supset \{c+E_{\omega}\}\supset \{2c+E_\omega\}\supset\dots\supset \{qc+E_\omega\}=E_\omega.
		\]
		Therefore $E_\omega$ is a periodic set with period $1/q$ and each number of form $\{p/q\}$ with $0\leq p\leq q$ belongs to the set of limit points of $\{x_{L_i}(\omega)\}_{i\geq 1}$. In particular, $0$ is a limit point which is not isolated.
	\end{proof}
	For a general sequence $\omega$ it is possible that $0$ is not a limit point of $\{x_i(\omega)\}_{i\geq 1}.$ To give such an example, let $\alpha=\pi/3,\beta=e/4.$ Let $\epsilon\in (0,0.5|\alpha-\beta|)$ be an arbitrarily chosen real number. Let $x_0=0$ and for each $i\geq 1$, whenever $\{x_{i-1}+\alpha\}\notin (-\epsilon,\epsilon)$ we set $x_i=\{x_{i-1}+\alpha\}$, or else we set $x_i=\{x_{i-1}+\beta\}.$ We record the sequence of 
	rotations by $\alpha$ or $\beta$ for each step. In such a way we have defined a $\alpha,\beta$ sequence $\omega$ with the property that $0$ is not limit point of $\{x_i(\omega)\}_{i\geq 1}.$ Therefore the recurrence property  plays a crucial role.
	
	Now we want to deal with finitely recurrent sequences. A trivial example of a finitely recurrent sequence is the constant sequence. In our setting, constant sequence corresponds to $\{i\alpha\}_{i\geq 1}$ with an irrational number $\alpha.$ In this way, our next result can be seen as a natural generalization of the topological minimality of irrational rotations.
	
	\begin{proof}[Proof of Theorem \ref{Th2}]
		Here we adopt all notations used in the proof of Theorem \ref{Zero}. We assume that $\{x_i(\omega)\}_{i\geq 1}$ is not dense in $[-0.5,0.5].$ According to the argument in the proof of Theorem \ref{Zero}, we see that it is of no loss of generality if we assume that $\lim_{i\to\infty} x_{L_i}(\omega)=0.$ In fact, we consider all limit points of $\{x_{L_i}(\omega)\}_{i\geq 1}.$ If there is an irrational number among those limit points then $E_\omega$ must be the whole interval. If a non-zero rational number $\{p/q\}$ is in the set of limit points, then $E_\omega$ is periodic of period $1/q.$ In this case we may replace $\alpha$ by $q\alpha$ and $\beta$ by $q\beta$ and obtain a new finitely recurrent sequence $\omega'$. Then we see that $E_\omega'$ is a scaled version of $E_\omega\cap [-0.5/q,0.5/q].$ We can repeat this argument by looking at the limit points of $\{x_{qL_i}(\omega')\}_{i\geq 1}.$ If it still contains non-zero rational numbers say $\{p'/q'\}$ then $E_\omega'$ is periodic with period $1/q'$. This implies that $E_\omega$ has period $1/qq'.$ If this argument can be repeated infinitely many times $E_{\omega}$ must have an arbitrary small period and therefore it is the whole interval. The only chance to stop this repeated argument is that for some integer $q$, zero is the only limit point of $\{x_{qL_i}(\omega')\}_{i\geq 1}.$ We can assume that this happens from the beginning.
		
		For each $i\geq 1$ we have
		\[
		v_{i}=\omega_{i-1}a_{i}
		\]
		and
		\[
		v_{i+1}=\omega_{i}a_{i+1}=\omega_{i-1}a_i\omega_{i-1}a_{i+1}.
		\]
		From here we see that
		\[
		x_{L_{i+1}}(\omega)=\{2x_{L_{i}}(\omega)+\Delta(a_i,a_{i+1})\},
		\]
		where we write $l_i$ as the length of $a_i$ and 
		\[
		\Delta(a_i,a_{i+1})=x_{l_{i+1}}(a_{i+1})-x_{l_{i}}(a_{i}).
		\]
		As $x_{L_i}(\omega)\to 0$ we see that $\{\Delta(a_i,a_{i+1})\}\to 0$ and this is only possible if there exists an integer $N$ such that $a_M=a_{M+1}=a_{M+2}\dots.$ In this case $\omega$ is a periodic sequence as it is of the following form
		\[
		\omega=(a a_M a)a_M(aa_Ma)\dots. 
		\]
		Since $\alpha_1,\dots,\alpha_k$ are rationally independent, we see that $E_\omega$ contains an orbit of an irrational rotation and therefore $E_\omega$ must be $[-0.5,0.5].$ This contradicts to our assumption and the result follows.
	\end{proof}
	\begin{cor}
		There exists a sequence over $\{0,1\}$ such that its orbit closure under the left shift does not contain any finitely recurrent sequence.
	\end{cor}
	\begin{proof}
		Recall the discussion below the proof of Theorem \ref{Zero} and the result follows.
	\end{proof}
	
	We remark here that although it is easy to construct a recurrent sequence which is not finitely recurrent, it is tricky to explicitly construct a sequence without any finitely recurrent sequence in its orbit closure under the left shift.
	\section{proof of Theorem \ref{Th3}}\label{dio}
	For $k$ rationally independent numbers $\alpha_1,\dots, \alpha_k$, we want to study the set $E_\omega$ for any $\omega\in\Lambda^{\mathbb{N}}.$ Recall that $\Lambda=\{\alpha_1,\dots,\alpha_k\}$ is the digit set we are interested in. For any integer $t\geq 1,$ we decompose the interval $[-0.5,0.5]$ into $t$ almost disjoint closed intervals of length $t^{-1}.$ Here we say that two closed intervals are almost disjoint if their intersection contains at most one point. Denote the collection of those intervals as $\mathcal{I}_t.$ We shall construct a directed graph $G_t$ with vertex set $\mathcal{I}_t.$ For each $i\in\{1,\dots,k\},$ we draw a directed edge from $a\in\mathcal{I}_{t}$ to $b\in\mathcal{I}_t$ if and only if $(\alpha_i+a)\cap b\neq\emptyset.$ If $t$ is large enough we see that there is no edges that loops around any vertex. We complete the construction by drawing all edge with all $i\in\{1,\dots,k\}.$ When $t$ is large enough we see that there exists no pairs $a,b\in \mathcal{I}_t$ such that there are more than $1$ directed edges from $a$ to $b$. A simple observation yields that the in/out degrees of each vertex are at most $3k.$
	
	Given the set $E_\omega,$ for large enough integers  $t$ we can collect the intervals in $\mathcal{I}_t$ that intersects with $E_\omega.$ We call this collection $\mathcal{I}_t(\omega).$ Starting with any interval $a\in\mathcal{I}_t(\omega),$ since $E_\omega$ is the closure of a multi-rotation orbit we see that there is another interval $b\in\mathcal{I}_t(\omega)$ such that $a,b$ are connected (from $a$ to $b$) in the directed graph $G_t.$ Then there is yet another $c\in\mathcal{I}_t(\omega)$ such that  $b,c$ are connected with a directed edge. Continuing with this manner, we can find a sequence over $\mathcal{I}_t(\omega)$ such that each consecutive pairs in this sequence are connencted with a directed edge. We write this sequence as $g=g_1g_2g_3\dots.$ We are interested in \emph{primitive cycles} which are words of finite length in $g$ such that the first and last digit are the same and all the digits in between are different from each other. For example, a primitive cycle with digit $1,2,3,4,5$ is $12341.$ 
	
	To see the existence of primitive cycles in any sequence over finite digits, consider for now that the digit set is $\{1,\dots,n\}$ with $n\geq 2$ being an integer. Let $\gamma=\gamma_1\gamma_2\dots$ be an infinite sequence over $\{1,\dots,n\}.$ Then there exist two integers $k_1<k_2$ such that $\gamma_{k_1}=\gamma_{k_2}.$ Now we consider $\gamma_{k_1}\dots\gamma_{k_2}.$ This might not be a primitive cycle as there could exist two integers $k'_1,k'_2$ such that $k'_1-k_1, k'_2-k_2$ are not both $0$ and such that
	$
	\gamma_{k'_1}=\gamma_{k'_2}.
	$
	Then we may consider the word $\gamma_{k'_1}\dots\gamma_{k'_2}$, this word has length strictly smaller than that of $\gamma_{k_1}\dots\gamma_{k_2}.$ Continuing this argument we see that eventually we will end up with a primitive cycle.
	
	In the situation of $g,$ we see that $g$ has a primitive cycle of length $s$ least $2.$ We may assume that $g_1\dots g_s$ is a primitive cycle. By the construction of multi-rotation we see that there are $k$ non-negative integers $n_1,\dots,n_k$ such that $\sum_{i=1}^k n_i=s$ and such that
	\[
	\left\|\sum_{i=1}^k n_i\alpha_i\right\|\leq 2t^{-1}.\tag{*}
	\]
	The above inequality is well-studied as a subject called \emph{Diophantine approximation}. Define a function $\Phi_\Lambda$ by setting \[\Phi_{\Lambda}(s)=\min\left\{\left\|\sum_{i=1}^k n_i\alpha_i\right\|:(n_1,n_2,\dots,n_k)\in\mathbb{Z}_{\leq s}^k\setminus (0,0,\dots,0) \right\}\tag{**}\]
	for each integer $s.$ Here $\mathbb{Z}^k_{\leq s}$ is the set of $k$ tuples of integers $n_1,\dots,n_k$ such that $|n_1|,\dots,|n_k|\leq s.$ We see that $\Phi_\Lambda(s)$ is less or equal to the LHS of $(*).$ 
	
	In general we know that $\Phi_\Lambda(s)\leq 2s^{-k}$ for all $s\geq 1$ by Dirichlet's pigeonhole argument. On the other hand by Theorem \ref{SchmidtThm}, we know that when $\alpha_1,\dots,\alpha_k$ are rationally independent algebraic numbers, for any $\delta>0$ there is a number $C_\delta>0$ and we have for all $s\geq 1,$
	\[
	\Phi_\Lambda(s)\geq C_\delta s^{-k-\delta}.
	\]
	Thus we have shown that if $g$ contains a primitive cycle with length $s$ and integers $n_1,\dots,n_k$ such that $n_1+\dots+n_k=s$ as discussed in above then
	\[
	\Phi_{\Lambda}(s)\leq \left\|\sum_{i=1}^k n_i\alpha_i\right\|\leq 2t^{-1}.
	\]
	If $\alpha_1,\dots,\alpha_k$ are rationally independent algebraic numbers then we see that,
	\[
	C_\delta s^{-k-\delta}\leq \Phi_{\Lambda}(s)\leq \left\{\sum_{i=1}^k n_i\alpha_i\right\}\leq 2t^{-1}.
	\]
	This implies that
	\[
	s\geq \left(\frac{C_\delta}{2} t\right)^{1/(k+\delta)}.
	\]
	As we clearly have $\#\mathcal{I}_t(\omega)\geq s$, therefore we have the following inequality,
	\[
	\#\mathcal{I}_t(\omega)\geq \left(\frac{C_\delta}{2} t\right)^{1/(k+\delta)}.
	\]
	Since the above holds for all $t\geq 1$ and $\#\mathcal{I}_t$ is at most double the minimal amount of $t^{-1}$-intervals required to cover $E_\omega$ we see that
	\[
	\lbox E_\omega\geq \frac{1}{k+\delta}.
	\]
	As $\delta$ can be chosen arbitrarily we see that
	\[
	\lbox E_\omega\geq \frac{1}{k}.
	\]
	This proves the result.
	\section{Further discussions}\label{FUR}
	\subsection{ Complexities of sequences and multi-rotation orbits}
	In terms of symbolic dynamics there are several notions of recurrence. We have met two of them, namely recurrence and finitely recurrence. In fact, by Birkhoff's recurrent theorem, one can replace `recurrent sequence' with `uniformly recurrent sequence' in Theorem \ref{Th1}(2). Here, a uniformly recurrent sequence $\omega$ over a digit set $\Lambda$ is a recurrent sequence with the property that for each integer $n\geq 1,$ there is an integer $l_n$ such that the word $\omega(1)\dots\omega(n)$ occurs in $\omega$ with gap at most $l_n.$ That is to say, if \[\omega(i)\dots\omega(i+n-1)=\omega(1)\dots\omega(n)\] for some $i\geq 1$ then there is an integer $i+n-1\leq j\leq i+n-1+l_n$ such that
	\[
	\omega(j)\dots\omega(j+n-1)=\omega(1)\dots\omega(n).
	\]
	In other words, for each integer $n$, there is a number $N(n)$ such that $\omega(1)\dots\omega(n)$ appears in every $N(n)$ block of $\omega.$ In general, $N(n)$ can be much larger than $n.$
	We see that finitely recurrence is a stronger notion than uniformly recurrence in the sense that the numbers $l_n$ can be bounded from above by  $Ln$ with a constant $L$. If we consider being finitely recurrent is a notion of having low complexity then Theorem \ref{Th2} essentially says that the assertion of Conjecture \ref{feng} holds for multi-rotation orbits along sequences of low complexity. In symbolic dynamics there are different notions of complexity. A commonly used one is the cardinalities of the collection of words of different lengths. More precisely, for each integer $n\geq 1$, let $p_n(\omega)$ be the cardinality of the following set of $n$-blocks of $\omega,$
	\[
	\{\omega(i)\dots\omega(i+n-1): i\geq 1\}.
	\]
	We are interested in the growth rate of $p_n(\omega)$ for $n\to\infty.$ On one hand, we know that if $\omega$ is recurrent and $p_n(\omega)\leq n$ for some integer $n$ then $\omega$ is eventually periodic. That is to say, there is an integer $N$ such that the shifted sequence $S^N(\omega)$ is periodic. In this case we know that the multi-rotations orbit $E_\omega$ contains a translated copy of an irrational rotation and therefore $E_\omega=[-0.5,0.5].$ In fact, we can extend this result slightly. We recall the following terminologies.
	\begin{defn}
		A recurrent sequence $\omega$ over a digit set $\Lambda$ is said to be Sturmian if there is an integer $n\geq 1$ such that $p_n(\omega)\leq n+\#\Lambda-1.$
	\end{defn}
	
	\begin{defn}
		A recurrent sequence $\omega$ over a digit set $\Lambda$ is said to be balanced if there is an integer $N\geq 1$ such that for any digit $\lambda\in\Lambda$ and all integer $n\geq 1$ the number of $\lambda$ in any two words $\omega',\omega''$ of $\omega$ with length $n$ differs at most $N.$
	\end{defn}
	
	We have the following result, see \cite[Theorem 108]{B17}.
	\begin{thm}
		Let $\Lambda=\{0,1\}.$ Then a sequence $\omega\in\Lambda^{\mathbb{N}}$ is Sturmian if and only if it is balanced.
	\end{thm}
	
	From the proof of \cite[Theorem 1.5(i)]{FX18}, we see that $E_\omega$ contains intervals if $k=2$ and $\omega\in\{\alpha,\beta\}^{\mathbb{N}}$ with rational independent numbers $\alpha,\beta$ is a balanced sequence. Thus, we see that the assertion of Conjecture \ref{feng} holds when $k=2$ and $\omega$ is Sturmian.
	
	Observe that when $\omega$ is finitely recurrent then $p_n(\omega)=O(n).$ On the other hand, in general we know that $p_n(\omega)\leq \#\Lambda^n.$ So there is still a big gap between finitely recurrent sequences and general sequences in terms of their complexity. From the discussions above, there seems to be a curious connection between $E_\omega$ and complexity of $\omega.$ We pose in below the following Conjecture.
	
	\begin{conj}
		In the hypothesis of Conjecture \ref{feng} with $k=2.$ Let $\omega\in\Lambda^{\mathbb{N}}$ be a uniformly recurrent sequence. Consider numbers $p_n(\omega)$ for integers $n\geq 1.$ If $p_n(\omega)=O(n)$ then $E_\omega$ contains intervals.
	\end{conj}
	The assertion of the above conjecture is true if $p_n(\omega)\leq n+1$(Sturmian) or $\omega$ is finitely recurrent, which are both natural examples for sequences with $p_n(\omega)=O(n)$. In other words, if $E_\omega$ does not contain intervals then $\omega$ must be complicated in some sense.
	
	\subsection{Diophantine approximation with sign restrictions}
	In Section \ref{dio} we have related the box dimension of multi-rotation orbit to a Diophantine approximation problem. What we are interested in is the following problem. Given a set of $k\geq 1$ real numbers $\Lambda=\{\alpha_1,\dots,\alpha_k\}$ we defined the following function ($(**)$ in Section \ref{dio}),
	\[\Phi_{\Lambda}(s)=\min\left\{\left\|\sum_{i=1}^k n_i\alpha_i\right\|:(n_1,n_2,\dots,n_k)\in\mathbb{Z}_{\leq s}^k\setminus (0,0,\dots,0) \right\}.\]
	By the argument in Section \ref{dio}, if $\Phi_\Lambda(s)\geq s^{-\tau}$ for infinitely many integers $s$ and for a number $\tau>0,$ then all multi-rotation orbits with rotation parameters $\{\alpha_1,\dots,\alpha_k\}$ have upper box dimension at least $1/\tau.$ Notice that in Section \ref{dio}, for $\{\alpha_1,\dots,\alpha_k\}$ being algebraic numbers, we have $\Phi_{\Lambda}(s)\geq C_\delta s^{-k-\delta}$ for all $s\geq 1$ with a suitable constant $C_\delta>0.$ This allows us to deduce the lower bound of the lower box dimension of all multi-rotation orbits.
	
	On one hand, we know that if $\alpha_1,\dots,\alpha_k$ are rationally independent then $\Phi_{\Lambda}(s)\leq 2s^{-k}$ for all $s\geq 1.$ On the other hand, we want to know what happens if $\Phi_{\Lambda}(s)\leq s^{-k-\delta}$ for all $s\geq 1.$ In the case when $k=1$, it is simple to see that if $\Phi_\alpha(s)\leq s^{-1-\delta}$ for all $s\geq 1$, then $\alpha$ must be a rational number. For $k\geq 2,$ little is known. See \cite[Section 3]{W09} for more details. In particular, by \cite[Theorem 62]{W09}, we know that for almost all $(\alpha_1,\dots,\alpha_k)\in\mathbb{R}^k,$ for all $\delta>0,$ we have
	\[
	\Phi_{\Lambda}(s)\geq s^{-k-\delta}
	\]
	for infinitely many $s\geq 1.$ This implies the following result.
	\begin{thm}\label{LebAlmost}
		Under the hypothesis of Conjecture \ref{feng}, there is a full Lebesgue measure set $A\subset\mathbb{R}^k$ such that whenever $(\alpha_1,\dots,\alpha_k)\in \mathbb{R}^k$ we have for all $\omega\in \Lambda^{\mathbb{N}}$ that
		\[
		\ubox E_\omega\geq \frac{1}{k}.
		\]
	\end{thm} 
	The above result is not as effective as Theorem \ref{Th3} in the sense we cannot tell whether a given $k$-tuple $(\alpha_1,\dots,\alpha_k)$ belongs to $A.$ However we can write down the complement $A^c$ explicitly as follows,
	\[
	A^c=\{(\alpha_1,\dots,\alpha_k): \Lambda=\{\alpha_1,\dots,\alpha_k\}, \exists \delta>0, \forall s\geq 1,\Phi_\Lambda(s)\leq s^{-k-\delta}\}.
	\]
	We know that if $(\alpha_1,\dots,\alpha_k)\in A^c$ then at least one of $\alpha_1,\dots,\alpha_k$ must be transcendental.
	
	We note here that the function $\Phi_{.}(.)$ is in fact an overkill to the problem on the box dimension of multi-rotation orbits. What we really want to consider is the following function
	\[
	\phi_\Lambda(s)=\min\left\{ \left\|\sum_{i=1}^k n_i\alpha_i\right\|: \sum_{i} n_i\leq s, n_1> 0,\dots,n_k> 0      \right\}.
	\]
	In general we see that $\Phi_\Lambda(s)\leq\phi_\Lambda(s)$ for all $s\geq 1.$ This type of problem was initiated by Schmidt, see \cite{S76}. In particular, it was shown that for $\alpha_1,\alpha_2$ being two rationally independent numbers then for any $\epsilon>0$, there exist infinitely many pairs of positive integers $k_1,k_2$ such that
	\[
	\|k_1\alpha_1+k_2\alpha_2\|\leq \epsilon\max\{k_1,k_2\}^{-\gamma}
	\]
	where $\gamma=(\sqrt{5}+1)/2$ is the golden ratio. It is known recently (see for example \cite{R14}) that this number $\gamma$ is optimal in the sense that for each $\delta$ there exists a pair of numbers $\alpha_1,\alpha_2$ such that
	\[
	\|k_1\alpha_1+k_2\alpha_2\|\leq \epsilon\max\{k_1,k_2\}^{-\gamma-\delta}
	\]
	holds for only finitely many positive integer pairs $k_1,k_2.$ In particular, this implies that exist a pair of rationally independent numbers $\alpha_1,\alpha_2$ and any multi-rotation orbits along sequences over $\{\alpha_1,\alpha_2\}$ has lower box dimension at least $0.618$. When there are more than two rotations, by \cite[Remark F]{S76}, for any $k\geq 3$, there exist $k\geq 3$ rationally independent numbers $\alpha_1,\dots,\alpha_k$ such that for any $\delta>0$ we have a positive number $C_\delta>0$ and for all $s\geq 1$,
	\[
	\phi_\Lambda(s)\geq C_\delta s^{-2-\delta},
	\] 
	where $\Lambda=\{\alpha_1,\dots,\alpha_k\}.$ This implies that for this particular choice of $\Lambda,$ all multi-rotation orbits have lower box dimension at least $0.5.$ 
	
	\section{Acknowledgement}
	HY was financially supported by the University of St Andrews.
	
	\providecommand{\bysame}{\leavevmode\hbox to3em{\hrulefill}\thinspace}
	\providecommand{\MR}{\relax\ifhmode\unskip\space\fi MR }
	\providecommand{\MRhref}[2]{%
		\href{http://www.ams.org/mathscinet-getitem?mr=#1}{#2}
	}
	\providecommand{\href}[2]{#2}

\end{document}